\documentclass[a4paper]{gtart}
\usepackage[inner=20mm, outer=20mm, textheight=245mm]{geometry}

\usepackage{psfrag, graphicx, subfigure, epsfig,color}
\usepackage{amsmath, amssymb, latexsym, euscript}
\usepackage{mathptmx, wrapfig, mathrsfs}
\usepackage[colorlinks,citecolor=blue,linkcolor=blue, backref=page]{hyperref}
\usepackage[nameinlink]{cleveref}
\usepackage[margin=1cm]{caption}

\theoremstyle{plain}
\newtheorem{theorem}{Theorem}
\newtheorem*{theorem*}{Theorem}

\newtheorem{proposition}[theorem]{Proposition}
\newtheorem{corollary}[theorem]{Corollary}

\theoremstyle{definition}

\newtheorem*{definition*}{Definition}

\theoremstyle{remark}

\numberwithin{equation}{section}

\def\lst{\operatorname{LST}}
\def\tri{\mathcal{T}}
\def\surface{\mathcal{S}}
\def\longitude{\lambda}
\def\meridian{\mu}
\def\normmin{S}
\def\manifold{M}
\def\Z{\mathbb{Z}}

\def\class{\alpha}

\begin{document}

\title{A new family of minimal ideal triangulations\\ of cusped hyperbolic 3--manifolds}
\author{J. Hyam Rubinstein, Jonathan Spreer and Stephan Tillmann}

\begin{abstract} 
Previous work of the authors with Bus Jaco determined a lower bound on the complexity of cusped hyperbolic $3$--manifolds and showed that it is attained by the monodromy ideal triangulations of once-punctured torus bundles. This paper exhibits an infinite family of minimal ideal triangulations of Dehn fillings on the link $8^3_9$ that also attain this lower bound on complexity.
\end{abstract}
	
\makeshorttitle


\section{Introduction}

We refer to \cite{Jaco-ideal-2020} for background and precise definitions used in this paper and only give a quick summary here. For a cusped orientable hyperbolic 3--manifold $\manifold$ of finite volume, the \textbf{norm} $|| \ \class \ ||$ of a non-trivial class $\class \in H_2 (\manifold,\Z_2)$ is the negative of the maximal Euler characteristic of a properly embedded surface $S$ (no component of which is a sphere) representing the class. Any surface $S$ with $[S]=\class$ and $-\chi(S) = || \ \class \ ||$ is \textbf{taut} if no component of $S$ is a sphere or a torus.
A class $\class \in H_2 (\manifold,\Z_2)$ determines a labelling of the ideal edges of an ideal triangulation of $\manifold$ by elements in $\Z_2 = \{0,1\}$. These can be interpreted as edge weights of a normal surface, called the \textbf{canonical normal representative} of $\class.$ 

The \textbf{complexity} $c(\manifold)$ is the minimal number of ideal tetrahedra in an ideal triangulation of $\manifold.$
In \cite[Theorem 1]{Jaco-ideal-2020}, the authors and Bus Jaco obtain the following lower bound on the complexity of a cusped hyperbolic $3$-manifold:

\begin{theorem}\hspace{-0.2cm}\emph{\cite{Jaco-ideal-2020}}
  \label{thm:sumofnorms}
  Let $\tri$ be an ideal triangulation of the cusped orientable hyperbolic $3$--manifold $\manifold$ of finite volume. If $H \le H_2 (\manifold,\Z_2)$ is a rank $2$ subgroup, then the number of ideal tetrahedra $|\tri|$ satisfies
  $$|\tri| \geq c(M) \geq \sum \limits_{0 \neq \class \in H} || \ \class\ ||. $$
   Moreover, in the case of equality the triangulation is minimal and each canonical normal representative of a non-zero element in $H\le H_2 (\manifold,\Z_2)$ is taut and meets each tetrahedron in a quadrilateral disc.
\end{theorem}

It is shown in \cite{Jaco-ideal-2020} that the above bound implies that the monodromy ideal triangulations of the hyperbolic once-punctured torus bundles are minimal. This note presents another infinite family of once-cusped hyperbolic $3$-manifolds $\manifold_{k,n}$,  admitting triangulations $\tri_{k,n}$ which achieve the lower bound in \Cref{thm:sumofnorms}. These manifolds are shown to arise from Dehn surgery on the link $8^3_9.$ What makes these triangulations interesting is that their structure is different from the structure of the monodromy ideal triangulations, and that they do not feature layered solid tori as Dehn fillings (as, for instance, the canonical triangulations exhibited by Gu\'eritaud and Schleimer~\cite{Gueritaud-canonical-2010}).
We now state the main theorem of this paper (the relevant definitions and an outline of the proof are given in \Cref{sec:overview}):

\begin{theorem}
 \label{thm:new-family}
Let $k, n\geq 3$ be odd integers.
The identification space $\manifold_{k,n}$ of the triangulation $\tri_{k,n}$ is hyperbolic and homeomorphic to the manifold obtained from the complement of $8^3_9$ by Dehn filling the symmetric cusp with slope $(-k,1)$ and one of the other cusps with slope $(n+1,1)$. Here, the natural geometric framing is used for the cusps. The triangulation $\tri_{k,n}$ achieves equality in 
\Cref{thm:sumofnorms} with $H = H_2 (\manifold_{k,n},\Z_2)\cong\Z_2\oplus\Z_2$, and, in particular, $\manifold_{k,n}$ has complexity $k+n.$
\end{theorem}

This theorem, and its proof, have the following consequences.

First, our approach to compute the $\mathbb{Z}_2$--norm introduces new techniques that extend the tool kits from \cite{Jaco-norm-2020, Jaco-ideal-2020, Jaco-coverings-2011} and is also applicable to the computation of the Thurston norm of Dehn fillings.

Second, the construction of the minimal triangulations generalises. We do not provide full details, and simply list a further family of triangulations that we conjecture to be minimal in \Cref{sec:constructions}.

Third, it remains an open problem to classify all ideal triangulations of cusped hyperbolic 3--manifolds that achieve the lower bound in \Cref{thm:sumofnorms}, and to provide a complete list of all underlying cusped hyperbolic 3--manifolds. In contrast, an analogous bound for the complexity of closed 3--manifolds and a complete characterisation of all tight examples are given in \cite{Jaco-norm-2020}.

\textbf{Acknowledgements.} 
Research of the authors is supported in part under the Australian Research Council's Discovery funding scheme (project number DP190102259). 


\section{The new infinite family}
\label{sec:overview}

The first member, $\tri_{3,3}$, of the infinite family $\tri_{k,n}$ was found by experimentation. Of the $9\,596$ minimal triangulations of $4\,815$ cusped hyperbolic $3$--manifolds contained in the orientable cusped hyperbolic census (shipped, for instance, with \texttt{Regina}~\cite{regina}), there are exactly nine manifolds for which the lower bound in 
\Cref{thm:sumofnorms} is attained. Eight of these are monodromy ideal triangulations of once-punctured torus bundles, but the remaining one, a triangulation of manifold \texttt{s781}, is not. 

One can check that this minimal ideal triangulation $\tri_{3,3}$ of \texttt{s781} achieves the lower bound as follows.
It has isomorphism signature \texttt{gLLMQbeefffehhqxhqq} and six ideal tetrahedra. The underlying hyperbolic manifold $\manifold_{3,3}$ has one cusp and $H_1 (\manifold_{3,3}, \mathbb{Z}) \cong \mathbb{Z} \oplus \mathbb{Z}_2 \oplus \mathbb{Z}_4.$
The space of all closed normal surfaces has 11 embedded fundamental surfaces consisting of one vertex linking torus, seven orientable surfaces of Euler characteristic at most $-2$ and three non-orientable surfaces of Euler characteristic equal to $-2.$ The non-orientable surfaces represent distinct non-trivial classes in $H_2 (\manifold_{3,3}, \mathbb{Z}_2).$ Since each even torsion class has a taut representative that is isotopic to a normal surface, and all non-trivial fundamental normal surfaces in $\tri_{3,3}$ have Euler characteristic $\leq -2$, it follows that $\tri_{3,3}$ attains the lower bound in \Cref{thm:sumofnorms}.

We now briefly describe how this example is extended to the infinite family of triangulations $\tri_{k,n}$, $k,n \geq 3$ odd. 

The ideal triangulation $\tri_{k,n}$ is build from a $k$--tetrahedron solid torus $T_k$ and an $n$--tetrahedron solid torus $T_n$, $k, n\geq 3$ odd, each with two punctures on the boundary, by identifying the eight boundary triangles (four per solid torus) in pairs. The identification space $\manifold_{k,n}$ has the required torsion classes in homology, namely $H_1 (\manifold_{k,n}, \mathbb{Z}) \cong \mathbb{Z} \oplus \mathbb{Z}_2 \oplus \mathbb{Z}_{k+1}$ (note that $ H_1 (\manifold_{k,n}, \mathbb{Z})$ does not depend on $n$). Hence we have $H_2 (\manifold_{k,n}, \mathbb{Z}_2) \cong \mathbb{Z}_2\oplus \mathbb{Z}_2$. We remark that we do not consider homology relative to the boundary. The triangulations of the solid tori are described in \Cref{sec:torus}, and the triangulations $\tri_{k,n}$ in \Cref{sec:triangulations}. 

\begin{figure}[hb]
 \centerline{\includegraphics[width=5cm]{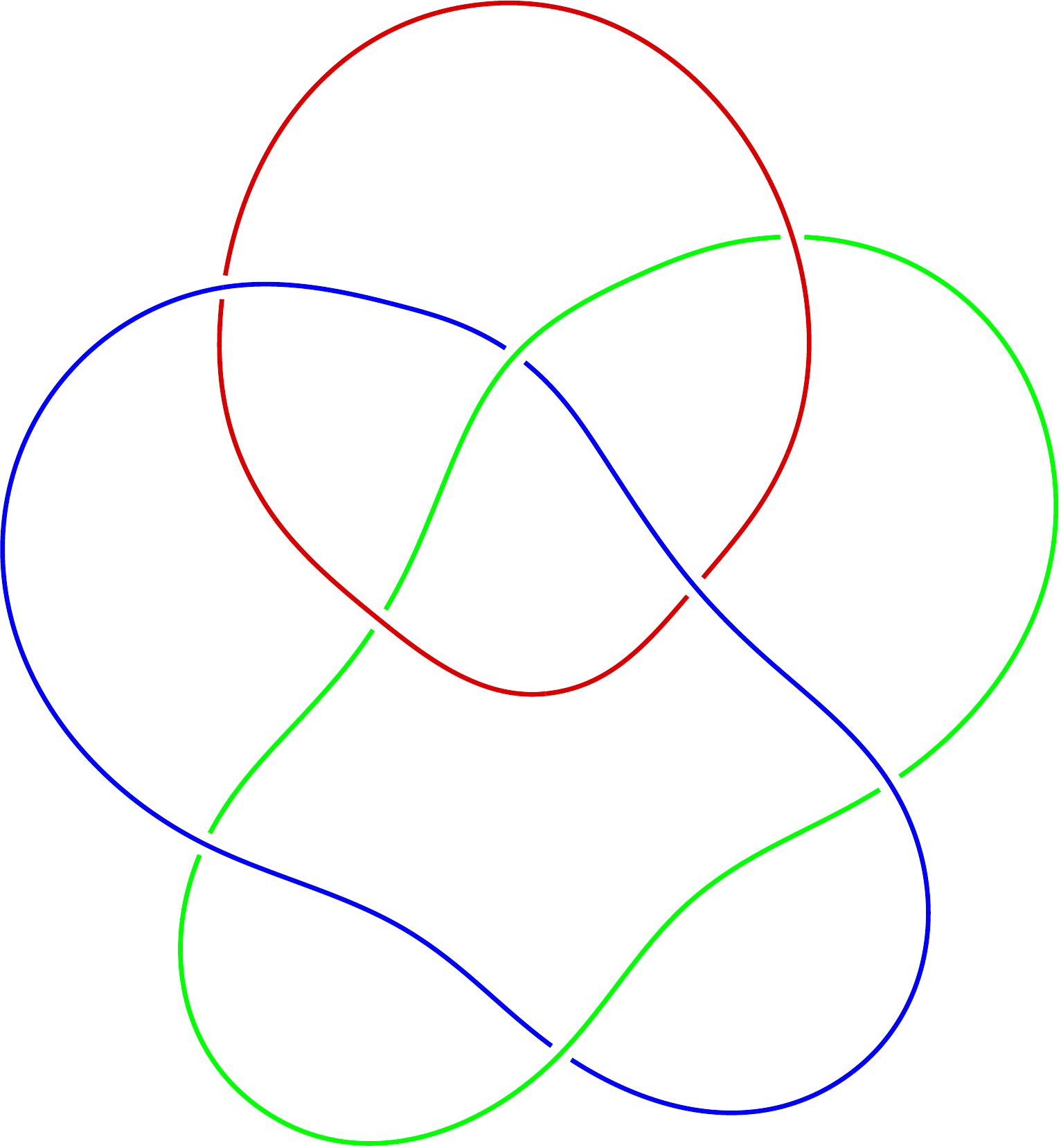}}
 \caption{Three component hyperbolic (non-alternating) link $8^3_9$ (Alexander-Briggs (Rolfsen) Number~\cite[Appendix C]{rolfsen2003knots}), L8a16 (Thistlethwaite Number~\cite{katlas}). Picture from \cite{katlas}. \label{fig:link}}  
\end{figure}

In order to apply \Cref{thm:sumofnorms}, we need to show that each $\manifold_{k,n}$ is hyperbolic, i.e.\thinspace has a complete hyperbolic structure of finite volume. To this end, we show in \Cref{sec:manifolds} that the manifolds $\manifold_{k,n}$ are obtained by Dehn surgery on the link $8^3_9$  and determine the surgery coefficients.
The complement of the three component link $8^3_9$, shown in \Cref{fig:link}, is hyperbolic and admits a symmetry which fixes one of the cusps and interchanges the other two. We refer to the former as the {\em symmetric} cusp, and it is shown in \textcolor{red}{red} in \Cref{fig:link}.
The manifolds $\manifold_{k,n}$ are all obtained from the complement of $8^3_9$ by filling two of its cusps, where one of these must be the symmetric cusp. An application of Lackenby's combinatorial word hyperbolic Dehn surgery theorem~\cite[Theorem 4.9]{Lackenby-word-2000} shows that $\manifold_{k,n}$ is hyperbolic for all $k,n \geq 17.$ A computation with Regina~\cite{regina} extends this to all $k,n \geq 3.$

We next determine the norm of the non-trivial classes in $H_2 (\manifold_{k,n}, \mathbb{Z}_2).$ 
This is achieved in \Cref{sec:mainresult} by using a different triangulation of $\manifold_{k,n}$, namely a triangulation $\tri'_{k,n}$ that uses the description of $\manifold_{k,n}$ as Dehn surgery on the link $8^3_9.$ 
This allows us to analyse the space of \emph{all} normal surfaces representing the homology classes. Namely, it reduces the argument to a calculation of all normal surfaces with a specific boundary pattern in the link complement, together with an analysis of how these surfaces extend to the surgery tori.

This completes the outline of the proof of \Cref{thm:new-family}. Below is a quick guide to the notation used for the various triangulations in this paper:
\begin{align*}
&\tri  \;\;\;\; \text{ideal triangulation of the complement of the link } 8^3_9 \text{ in } \mathbb{S}^3 \\
&\tri'  \;\;\; \text{triangulation of the exterior of the link } 8^3_9  \text{ in }  \mathbb{S}^3 \text{ with one boundary component removed}\\
&\manifold_{k,n}  \cong\text{Manifold('8\^{}3\_9')}(\; \infty\; , (n+1,-1), (k+1,1)) \\
&\tri_{k,n} = T_k \cup T_n \text{ ideal triangulation of } \manifold_{k,n} \\
&\tri'_{k,n} = \tri' \cup \lst(1,m) \cup \lst(1, k-1)  \text{ triangulation of } \manifold_{k,n}  \text{ with one ideal and two material vertices } 
\end{align*}


\section{An infinite family of solid tori}
\label{sec:torus}

The construction of our infinite families of minimal triangulations of once-cusped hyperbolic manifolds is based on a family of solid torus triangulations $T_m$. Start with a 2--triangle standard triangulation of the annulus. The two boundary curves of the annulus are edges in this triangulation, and each of the remaining three edges is incident with both boundary edges. One now iteratively layers tetrahedra on one side of this annulus, such that the initial triangulation of the annulus is in the boundary of the solid torus, and each interior edge has degree four. This completely characterises these triangulations combinatorially. In what follows, we mainly introduce notation to make this explicit and we highlight some features that will be relevant later.

Let $m\ge 1$ and $\Delta_1, \dots , \Delta_m$ be a collection of Euclidean 3--simplices. The vertices of each $\Delta_j$ are labelled $0, 1, 2, 3$, and for each subset $\Lambda \subseteq \{0, 1, 2, 3\},$ $\Delta_j(\Lambda)$ is the subsimplex spanned by $\Lambda.$ 
The triangulation $T_1$ is obtained by making the edge identification $\Delta_1 (03) = \Delta_1 (12).$ 
For $m>1,$ $T_{m}$ is obtained from $T_{m-1}$ by adding $\Delta_m$ via the face pairings:
\begin{align*}
& \Delta_{m} (102) \,\, \mapsto \,\, \Delta_{m-1} (013), \\
& \Delta_{m} (023) \,\, \mapsto \,\, \Delta_{m-1} (023), \; \text{if $m$ is even},\\
& \Delta_{m} (123) \,\, \mapsto \,\, \Delta_{m-1} (123), \; \text{if $m$ is odd}.
\end{align*}
If $m$ is even, this is a layering of $\Delta_{m}$ on the edge $\Delta_{m-1} (03)$; and if $m$ is odd, this is a layering of $\Delta_{m}$ on the edge $\Delta_{m-1} (13).$

It follows from the construction that the faces $\Delta_1(012)$ and $\Delta_1(023)$ form an annulus in the boundary of $T_{m}$, and that for $m>1,$ $T_m$ is a solid torus obtained from $T_{m-1}$ by layering a tetrahedron on an edge not contained in the boundary faces $\Delta_1(012)$ and $\Delta_1(023)$. The different pairings for $m$ odd or even ensure that each interior edge of $T_m$ has degree four.

The two boundary edges of $T_m$ corresponding to a longitude of the solid torus are $\Delta_1(01) = \ldots = \Delta_m(01)$ and $\Delta_1(23) = \ldots = \Delta_m(23)$. We denote $\Delta_1(01) = \ldots = \Delta_m(01)$ by $\longitude$, see \Cref{fig:meridian}.

We note that $T_m$ has two vertices $\Delta_1(0) = \Delta_1(1) = \ldots = \Delta_m(0) = \Delta_m(1) $ and $\Delta_1(2) = \Delta_1(3) = \ldots = \Delta_m(2) = \Delta_m(3)$. 

The boundary of the solid torus $T_m$ is a 4--triangle, 2--vertex torus. For $m$ odd this results in boundary faces $\Delta_m (013),$ $\Delta_m (123),$ $\Delta_1 (012),$ and $\Delta_1 (023),$ see also \Cref{fig:meridian}.

\begin{figure}[p]
 {\includegraphics[width=0.91\textwidth]{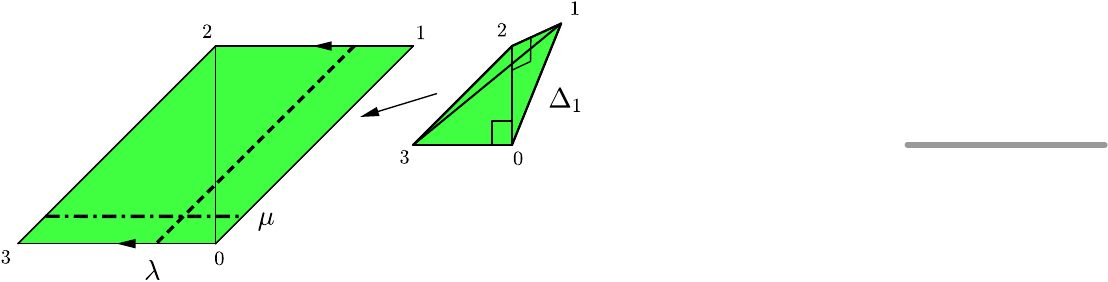}}

 \vspace{0.1cm} \hrule \vspace{0.1cm}

 {\includegraphics[width=0.91\textwidth]{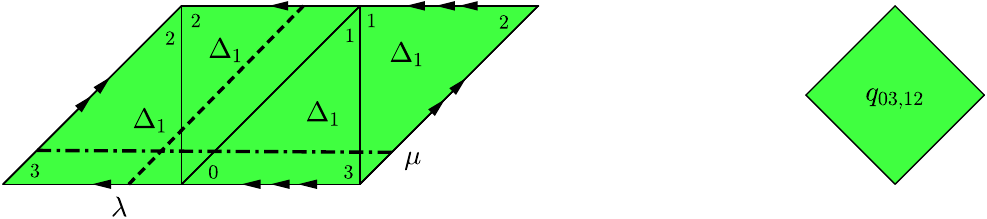}}

 \vspace{0.1cm} \hrule \vspace{0.1cm}

 {\includegraphics[width=0.91\textwidth]{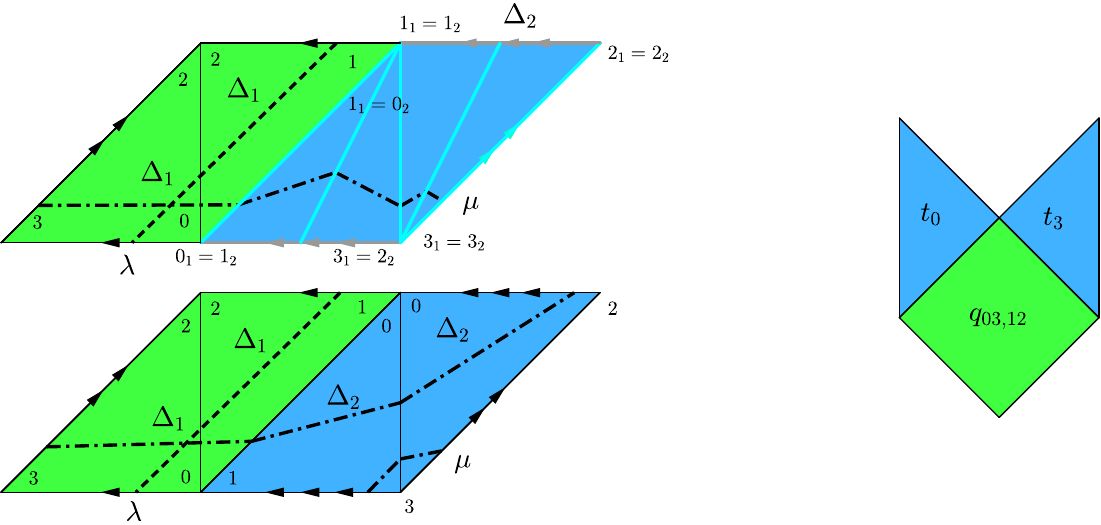}}
 
 \vspace{0.1cm} \hrule \vspace{0.1cm}
 
 {\includegraphics[width=\textwidth]{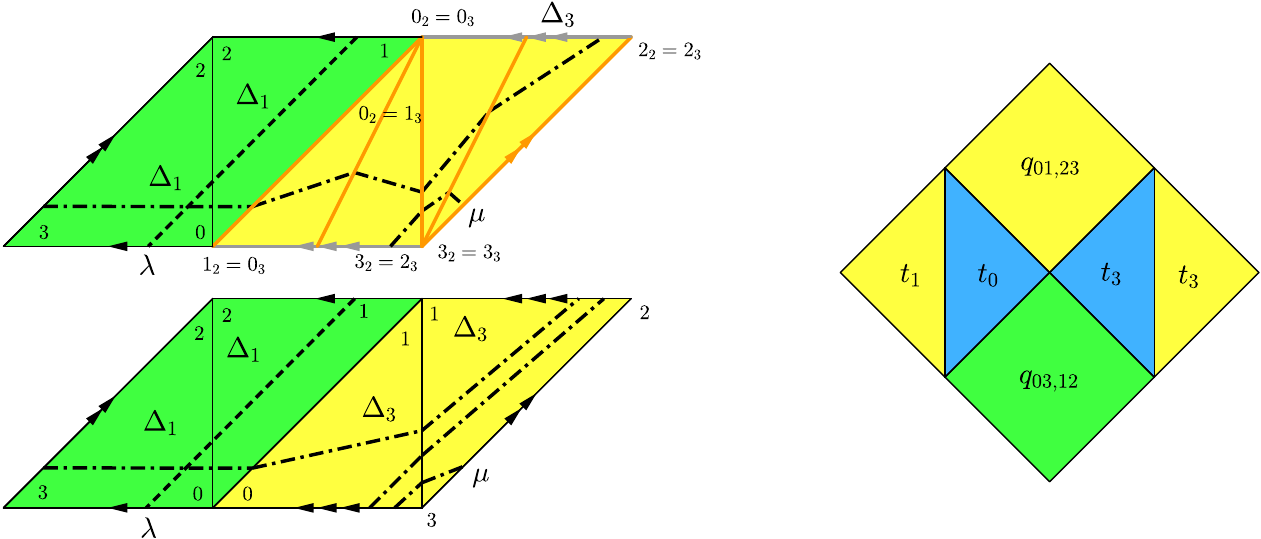}}
 \caption{Construction of $T_m$, longitude, meridian, and meridian disc. Top: the base annulus. Second, third, and fourth row: $T_1$, $T_2$, and $T_3$ respectively. In each step, another tetrahedron is layered onto the existing triangulation.  Right column: the normal pieces of the corresponding meridian disc. \label{fig:meridian}}
\end{figure}

In order to determine the meridian $\meridian$ of $T_m$ we trace the meridian disc of $T_m$ through the layerings:
\begin{itemize}
  \item In the 2--triangle annulus the seed of the meridian disc is a horizontal arc connecting the two boundary components (\Cref{fig:meridian} top row). 
  \item Layering $\Delta_1$ to the annulus adds a single horizontal quad of type $q_{03/12}$ in $\Delta_1$ to the seed of the meridian disc (\Cref{fig:meridian}, second row).
  \item Layering $\Delta_2$ on edge $\Delta_1 (03) = \Delta_1 (12)$ adds two normal triangles in $\Delta_2$ (one of type $t_0$ and one of type $t_3$) to the meridian disc (\Cref{fig:meridian}, third row).
  \item Layering $\Delta_3$ on edge $\Delta_2 (02) = \Delta_2 (13)$, adds two normal triangles (of types $t_1$ and $t_3$ respectively) and one normal quadrilateral (of type $q_{01/23}$) in $\Delta_3$ to the meridian disc (\Cref{fig:meridian}, bottom row).
  \item Layering $\Delta_j$, $j \geq 4$, adds two more normal triangles (of types as before) and $m-2$ normal quadrilaterals of type $q_{01/23}$ to the meridian disc. See \Cref{fig:general} for the meridian disc in $T_m$ and its boundary curve. 
\end{itemize}

\begin{figure}[t]
 \centerline{\includegraphics[width=\textwidth]{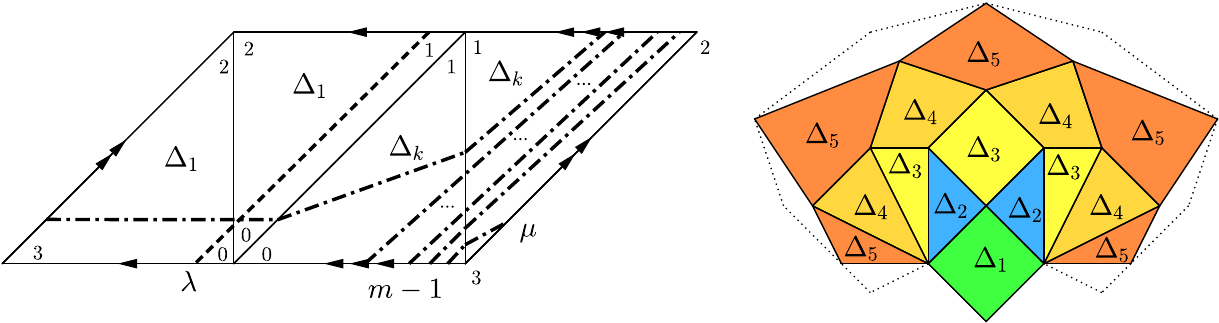}}
 \caption{Left: boundary of $T_m$ with longitude and meridian. Right: meridian disc of $T_m$ as a normal surface.  \label{fig:general}}
\end{figure}

The triangulation $T_m$ admits three normal surfaces $F_1$, $F_2$, and $F_3$ with exactly one normal quadrilateral per tetrahedron and every normal quadrilateral type features in exactly one of these three surfaces:
\begin{itemize}
  \item Starting with quadrilateral $q_{01/23}$ in $\Delta_1$ we obtain a separating annulus $F_1$ intersecting the boundary of $T_m$, $m\geq 2$ arbitrary, in two longitudes, denoted by $\gamma_1$.
  \item Starting with quadrilateral $q_{02/13}$ in $\Delta_1$ we obtain a one-sided (non-orientable) surface $F_2$ intersecting the boundary of $T_m$, $k\geq 3$ and odd, horizontally in curve $\gamma_2$.
  \item Starting with quadrilateral $q_{03/12}$ in $\Delta_1$ we obtain another one-sided (non-orientable) surface $F_3$ intersecting the boundary of $T_m$, $m\geq 3$ and odd, diagonally in curve $\gamma_3$.
\end{itemize}
See \Cref{fig:canonical} for an illustration.

\begin{figure}[hbt]
 \centerline{\includegraphics[width=\textwidth]{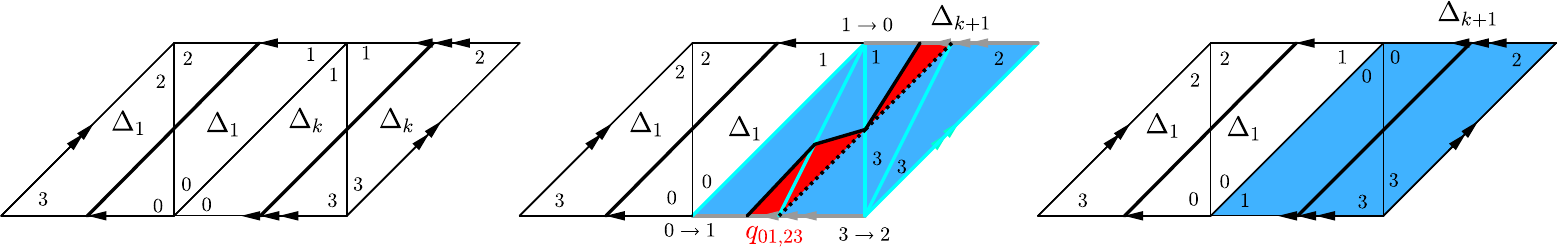}}

 \vspace{0.3cm} \hrule \vspace{0.3cm}

 \centerline{\includegraphics[width=\textwidth]{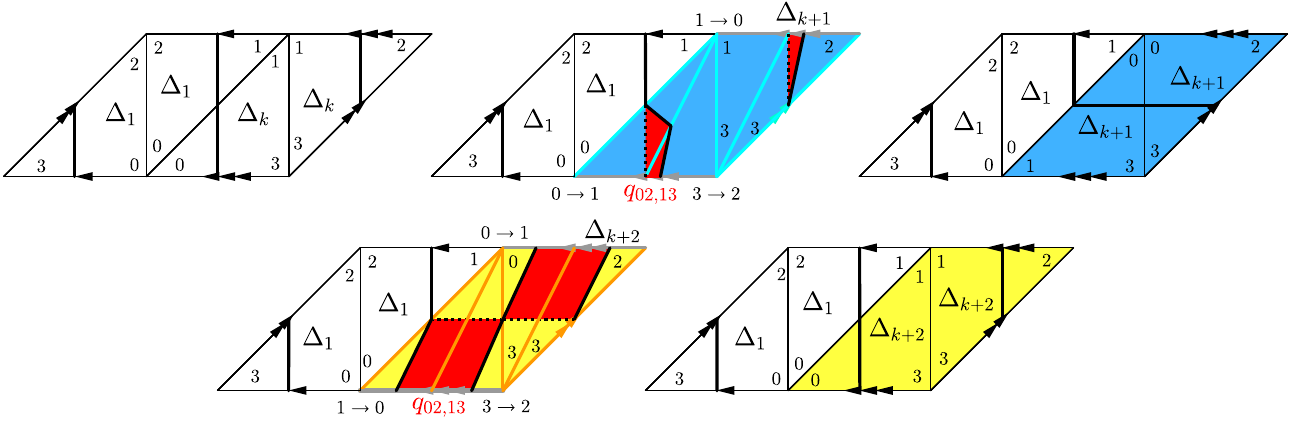}}

 \vspace{0.3cm} \hrule \vspace{0.3cm}

 \centerline{\includegraphics[width=\textwidth]{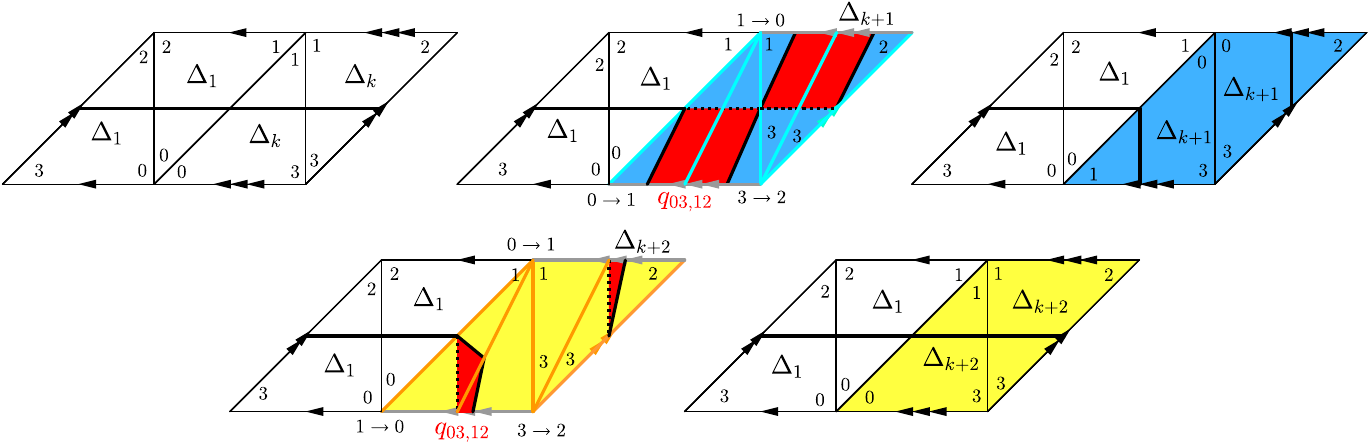}}
 
 \caption{Boundary patterns of normal surfaces $F_1$, $F_2$, and $F_3$ in $T_m$. The boundary pattern of $F_1$ does not depend on $m$, whereas the boundary patterns of $F_2$ and $F_3$ depend on the parity of $m$.  \label{fig:canonical}}
\end{figure}

Due to the following result, we know when and how a curve on the boundary of a solid torus extends to a one-sided incompressible properly embedded (connected) surface in its interior. 

\begin{proposition}[Corollary 2.2 in \cite{Frohman-one-sided-1986}]
  \label{prop:boundingcurve}
Suppose $T$ is a solid torus, $\meridian$ a standard meridian and $\longitude$ a longitude for $T.$  
  A one-sided connected incompressible surface $S \subset T$ has boundary a single curve homotopic with $2p \longitude + q \meridian$, for some $p, q\in \mathbb{Z}$ satisfying $p\neq 0$, $2p$ and $q$ co-prime. Conversely, every simple closed curve on $\partial T$ homotopic with $2p \longitude + q \meridian$, where $p\neq 0$, $2p$ and $q$ co-prime, is the boundary of a one-sided incompressible surface in $T$.
  
Moreover, the negative Euler characteristic of $S$ is determined by $[\partial S] = 2p \longitude + q \meridian$ and given by the number $N(2p,q)$ defined by $N(2p,q)=N(2(p-Q),q-2m)+1$, where $Q$ is the smallest positive integer so that there exists an integer $m$ such that $Qq =2pm \pm 1$ and $N(2p,1)=p-1$.
\end{proposition}

\begin{corollary}
\label{cor:chi}
Let $m\ge 3$ be odd.
The surface $F_2$ with boundary $\gamma_2$ is an incompressible one-sided surface with negative Euler characteristic $(m-1)/2$ properly embedded in $T_m$, and the surface $F_3$ with boundary $\gamma_3$ is an incompressible one-sided surface with negative Euler characteristic $(m-3)/2$ properly embedded in $T_m.$ 
\end{corollary}

\begin{proof}
We have 
\[\gamma_i \ \simeq\ \langle \gamma_i, \meridian \rangle \longitude + \langle \gamma_i, \longitude \rangle \meridian. \]
As can be deduced from \Cref{fig:canonical}, this results in $\gamma_2 = (m+1) \longitude + \meridian$ and $\gamma_3 = (m-1) \longitude + \meridian$.

\Cref{prop:boundingcurve} applied to $\gamma_2$ and $\gamma_3$ shows that these curves bound incompressible surfaces of negative Euler characteristic $(m-1)/2$ and $(m-3)/2$ respectively. Now $F_2$ and $F_3$ have precisely this negative Euler characteristic. Hence they must be incompressible. 
\end{proof}


\section{The family of triangulations}
\label{sec:triangulations}

Suppose $k, n\geq 3$ odd.
Let $T^1 = T_k^1=T_k$ and $T^2 = T_n^2=T_n$ be triangulations of the solid torus from \Cref{sec:torus}. Denote their tetrahedra by $\Delta^i_j,$ $i=1,2.$ 
Glue the boundary of $T^1$ to the boundary of $T^2$ in the following way:
\begin{equation}
  \label{eq:interface}
\Delta^1_1(012)  \mapsto  \Delta^2_{1}(120); \,\ \Delta^1_k(013)_1 \mapsto \Delta^2_{1}(032); \,\ \Delta^1_1(023) \mapsto  \Delta^2_{n}(321); \,\ \Delta^1_k(123) \mapsto  \Delta^2_{n}(130).
\end{equation}
We denote the set the face pairings in \eqref{eq:interface} by $\Phi.$

This identifies all vertices from both boundary tori, and the link of this single vertex is a torus. The identification space of the resulting triangulation $\tri_{k,n}$ is an orientable pseudo-manifold having the single vertex as the only non-manifold point, and the complement of the vertex is denoted $\manifold_{k,n}.$ See \Cref{fig:gluing} for the two torus boundaries of $T^1$ and $T^2$. 

We refer to the image of the boundaries of $T^1$ and $T^2$ in $\tri_{k,n}$ as the {\em gluing interface}, see \Cref{fig:interface in link} for a drawing of the associated abstract $2$-complex. As observed, this has four triangles, one vertex, and we claim that there are four edges.
For the twenty edges involved in the identifications on the boundaries of the tori, we have identifications $\Delta_1^1(01)=\Delta_k^1(01)$, $\Delta_1^1(23)=\Delta_k^1(23)$, $\Delta_1^1(03)=\Delta_1^1(12)$ and $\Delta_k^1(03)=\Delta_k^1(12)$ coming from the gluings inside $T^1$; and similarly
$\Delta_1^2(01)=\Delta_n^2(01)$, $\Delta_1^2(23)=\Delta_n^2(23)$, $\Delta_1^2(03)=\Delta_1^2(12)$ and $\Delta_n^2(03)=\Delta_n^2(12)$ coming from the gluings inside $T^2$.  Combining this with the identifications of the boundaries yields four edge classes in the gluing interface:
\begin{align*}
e_1 = & \Delta_1^1(01)=\Delta_k^1(01) = \Delta_1^2(03)=\Delta_1^2(12), \\
e_2 = & \Delta_1^1(23)=\Delta_k^1(23) = \Delta_n^2(30)=\Delta_n^2(21), \\
e_3 = & \Delta_1^2(01)=\Delta_n^2(01) = \Delta_1^2(23)=\Delta_n^2(23) = \Delta_1^1(20)=\Delta_k^1(31), \\
e_4 = & \Delta_1^1(03)=\Delta_1^1(12) = \Delta_k^1(30)=\Delta_k^1(21) = \Delta_1^2(20)=\Delta_n^2(31).
\end{align*}

Denote the quadrilateral normal surfaces in the solid tori $T^i$ described in \Cref{sec:torus} by $F^i_j$, $i=1,2$, $j = 1,2,3$. Their edge weights $(w(e_1),w(e_2),w(e_3),w(e_4))$ in $\partial T^i$, $i=1,2$ are
\begin{itemize}
  \item $(1,1,0,1)$ for $F_2^1$ and $F_1^2$;
  \item $(0,0,1,1)$ for $F_1^1$ and $F_3^2$;
  \item $(1,1,1,0)$ for $F_3^1$ and $F_2^2$.
\end{itemize}

It follows that $\tri_{k,n}$ admits three normal surfaces 
$\surface_1 = F_2^1 \cup F_1^2$,
$\surface_2 = F_1^1 \cup F_3^2$, and 
$\surface_3 = F_3^1 \cup F_2^2$,
each consisting of a single quadrilateral per tetrahedron such that every normal quadrilateral type of $\tri_{k,n}$ occurs in exactly one of these surfaces.

Their respective Euler characteristics can be computed from \Cref{prop:boundingcurve}. For this, note that the boundary curves of the surfaces $F^i_j$ inside $T^i$, $i=1,2$, $j = 1,2,3$, are made up of four normal arcs and four midpoints of edges. 

In $\tri_{k,n}$, the surfaces $\surface_j$, $j=1,2,3$, still meet the gluing interface $T^1 \cap T^2$ in four normal arcs, but only three (resp.\thinspace two) midpoints of edges for $\surface_3$ and $\surface_1$ (resp.\thinspace $\surface_2$). This can be seen by looking at their edge weights for the four edges in the gluing interface.

Thus, the Euler characteristic of each surface pieces is the value given in \Cref{prop:boundingcurve} minus $1$ (resp.\thinspace minus $2$). Altogether we have 
\begin{align}
\chi(\surface_1) \ &= (\chi(F_2^1)-1) + (\chi (F_1^2)-1) + 1 &&= - \frac{k-1}{2} + 0 -1 &&= -\frac{k+1}{2}\\
  \chi(\surface_2) \ &= (\chi(F_1^1)-2) + (\chi (F_3^2)-2) + 2 &&= 0 - \frac{n-3}{2} - 2  &&= -\frac{n+1}{2}\\
  \chi(\surface_3) \ &= (\chi(F_3^1)-1) + (\chi (F_2^2)-1) + 1 &&= - \frac{k-3}{2} - \frac{n-1}{2} - 1  &&= - \frac{k+n-2}{2}  
\end{align}

Each of the $\surface_j$, $j=1,2,3$ is the canonical representative of one of the non-zero $\mathbb{Z}_2$--homology classes and we have 
\[
\chi(\surface_1) + \chi(\surface_2) + \chi(\surface_3) = -(k+n).
\]
In order to apply \Cref{thm:sumofnorms} to conclude that $\tri_{k,n}$ is minimal, we need to show that each $\manifold_{k,n}$ is hyperbolic (see \Cref{sec:manifolds}) and that each $\surface_j$ is taut (see \Cref{sec:mainresult}).


\section{Hyperbolicity and algebraic topology of $\manifold_{k,n}$}
\label{sec:manifolds}

We start by building a 3--cusped manifold using a variant of the construction of $\tri_{k,n}$ described above. 
For this, take a cone $\widehat{C}^1$ over $\partial T^1$ and a cone $\widehat{C}^2$ over $\partial T^2$, where $T^1 = T_k^1$ and $T^2 = T_n^2$ are the solid torus triangulations from \Cref{sec:triangulations} and $k, n\ge 3$ both odd.
Now we glue them together using the same face pairings $\Phi$ from \eqref{eq:interface} to obtain a triangulation $\tri$ with eight tetrahedra of a pseudo-manifold $\widehat{N}$ with three vertices having torus links. This is defined by the gluing table given in \Cref{fig:tableNk} and shown in \Cref{fig:ideal triangulation}. 

Using Regina \cite{regina} we can check that the complement $N = \widehat{N} \setminus \widehat{N}^{(0)}$ of the three vertices in $\widehat{N}$ is a 3--manifold homeomorphic to the complement of  the link $8^3_9$, and that this is a minimal triangulation of $N.$ We refer to $C^1 = \widehat{C}^1\setminus \widehat{N}^{(0)}$ and $C^2 = \widehat{C}^2\setminus \widehat{N}^{(0)}$ respectively as the \emph{(red) cusp 1} and the \emph{(blue) cusp 2} of $N.$ The red cusp is the symmetric cusp of the link complement, and the blue cusp therefore one of the non-symmetric cusps.

\begin{figure}[thb]
  \centerline{\includegraphics[width=\textwidth]{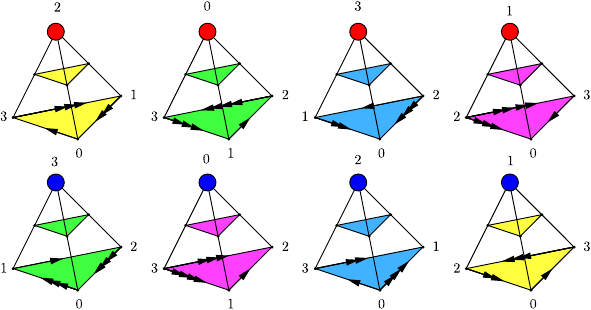}}
 \caption{The ideal triangulation of $N$, the complement of the link $8^3_9$. \label{fig:ideal triangulation}}
\end{figure}

\begin{figure}[thb]
  \centerline{\includegraphics[width=0.8\textwidth]{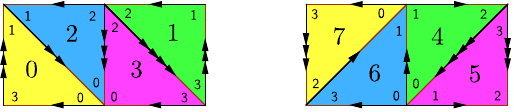}}
 \caption{The interface viewed from red and blue cusp\label{fig:interface in link}}
\end{figure}

  \begin{table}
   \begin{center}
    \begin{tabular}{r|rrrr}
      Tet & $(012)$ & $(013)$ & $(023)$ & $(123)$ \\
      \hline
$0$&$2 (013)$&${\bf 7 (023)}$&$1 (102)$&$1 (103)$\\
$1$&$0 (203)$&$0 (213)$&$3 (123)$&${\bf 4 (120)}$\\
$2$&${\bf 6 (130)}$&$0 (012)$&$3 (021)$&$3 (031)$\\
$3$&$2 (032)$&$2 (132)$&${\bf 5 (231)}$&$1 (023)$\\
      \hline
$4$&${\bf 1 (312)}$&$6 (012)$&$5 (130)$&$5 (120)$\\
$5$&$4 (312)$&$4 (302)$&$7 (123)$&${\bf 3 (302)}$\\
$6$&$4 (013)$&${\bf 2 (201)}$&$7 (013)$&$7 (012)$\\
$7$&$6 (123)$&$6 (023)$&${\bf 0 (013)}$&$5 (023)$\\
      \hline
    \end{tabular}
   \end{center}
   \caption{Gluing table for $3$-cusped manifold $N$, an ideal triangulation homeomorphic to the complement of $8^3_9$. Bold entries are gluings across the interface, all simplices are coherently oriented. \label{fig:tableNk}}
  \end{table}

It follows from the construction that the \emph{interface} between the red and blue cusps in $N$ is \emph{identical} to the interface between the solid tori in $\manifold_{k,n}.$ It is shown in \Cref{fig:interface in link}.
We now truncate the red and blue cusps along vertex linking surfaces. These surfaces are naturally combinatorially isomorphic with $\partial T^1$ and $\partial T^2$. 
The cusps are then filled with the solid tori $T^1_k$ and $T^2_n$ respectively, such that the gluing agrees with the natural combinatorial isomorphism (see 
\Cref{fig:filling}). With respect to the framing indicated in the figure, this results in the manifold
\[ N_{k,n}= N(\; \infty\; , (-k,1),(n+1,1)).\]
For the interested reader who wishes to use SnapPy to study these manifolds, we remark that 
\begin{equation}
\label{eq:snappy}
N_{k,n}=\text{Manifold('8\^{}3\_9')}(\; \infty\; , (n+1,-1), (k+1,1)).
\end{equation}
The \emph{Euclidean} equilateral triangles shown in \Cref{fig:filling} in fact give the Euclidean structures (up to scale) on the cusp cross sections with respect to the complete hyperbolic structure on $N.$ Here, each ideal tetrahedron has shape $i$ and each triangle therefore angles $\frac{\pi}{2}, \frac{\pi}{4},\frac{\pi}{4}$. The framing used in this paper is thus the natural geometric framing of the cusps.

\begin{figure}[thb]
  \centerline{\includegraphics[width=0.66\textwidth]{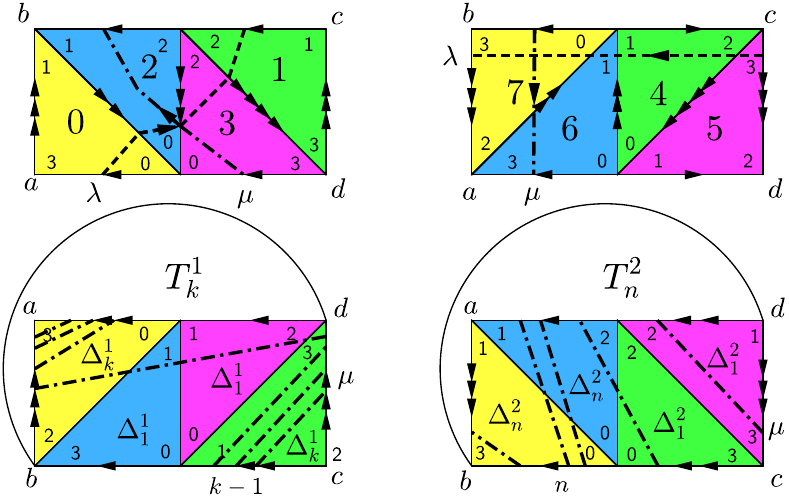} }
  \caption{Fillings of cusps of $\manifold$; indicated are the peripheral curves on the vertex links and the meridian curves on the solid tori. \label{fig:filling}}
\end{figure}

Casson observed that if an ideal triangulation admits an angle structure, then the manifold must admit a complete hyperbolic metric of finite volume (see~\cite[Theorem 10.2]{Rivin-combinatorital-2003}). Using this, computation with Regina~\cite{regina} confirms that $N_{k,n}$ is hyperbolic for all $17 \ge k, n \ge 3.$ The same result could also be obtained with SnapPy~\cite{SnapPy}.

To show that $N_{k,n}$ also has a complete hyperbolic metric of finite volume for all $k, n \ge 17,$ we note that due to Thurston's hyperbolisation theorem for Haken manifolds~\cite{Thurston-hyperbolic-1986, Otal-hyperbolization-1998, Kapovich-hyperbolic-2001}, it suffices to show that $N_{k,n}$ is irreducible, atoroidal, and not Seifert fibred. This conclusion is achieved by an application of Lackenby's combinatorial version of the Gromov-Thurston $2\pi$ theorem \cite[Theorem 4.9]{Lackenby-word-2000} (see also the sentence after its proof for the case when not all cusps are filled).
We refer to \cite{Lackenby-word-2000} for the definitions and statements required to understand the next paragraph.

The hypothesis of  \cite[Theorem 4.9]{Lackenby-word-2000} is achieved by applying \cite[Proposition 4.10]{Lackenby-word-2000}, which gives a criterion using angle structures. We use the angle structure on $N$ determined by the complete hyperbolic structure, where each triangle in the vertex link has angles $\frac{\pi}{2}, \frac{\pi}{4},\frac{\pi}{4}$. In particular, Lackenby's combinatorial length of any curve $\gamma$ on a vertex link is estimated from below by the shortest simplicial length $|\gamma|$ of a representative of $[\gamma]$ times $\frac{\pi}{8}.$ The theorem thus applies to any set of filling curves $\gamma$ with $|\gamma|\frac{\pi}{8} > 2 \pi,$ and hence $|\gamma| \ge 17.$ Examination of the vertex links gives the simple bounds
$|\meridian_1^{-k}\longitude_1| \ge k$ and $|\meridian_2^{n+1}\longitude_2| \ge n+1$. Therefore $k, n \ge 17$ are sufficient to ensure hyperbolicity.

Next, we observe that by construction the manifold $N_{k,n}$ is homotopy equivalent with $\manifold_{k,n},$ and that we may choose the homotopy to be the identity in a neighbourhood of the single cusp (corresponding to the vertex on the interface). Now take a sufficiently large compact core of each manifold, and double it along its boundary. Then we have two closed Haken 3--manifolds that are homotopy equivalent, and hence homeomorphic by Waldhausen~\cite{Waldhausen-irreducible-1968}. Since $N_{k,n}$ is hyperbolic and
homeomorphism preserves the JSJ decomposition, it follows that $N_{k,n}$ is homeomorphic with $\manifold_{k,n}.$ In particular, $\manifold_{k,n}$ is also hyperbolic.

It is clear from the description of the normal surfaces in the previous section that $\manifold_{k,n}=N_{k,n}$ has homology amenable to an application of \Cref{thm:sumofnorms} (\cite[Theorem 1]{Jaco-ideal-2020}).
This can be verified using a homology computation using the chosen peripheral framings. Using the triangulation, one computes
\[ \pi_1(N) = \langle \; a, b, c \; \mid \; abcbc^{-1}a^{-1}cb^{-1}c^{-1}b^{-1} , \; abc^{-1}ab^{-1}a^{-1}bca^{-1}b^{-1}\; \rangle\]
with the two peripheral subgroups of interest to us (in the geometric framing) generated by:
\[\meridian_1 = c^{-1}a, \quad \longitude_1 = b^{-1}a^{-1}bc, \qquad \meridian_2 = a, \quad \longitude_2 = bcbc^{-1}\]
This implies:
\[H_1(\manifold_{k,n}) \cong \langle \;a, b, c\; \mid \; (a^{-1}c)^{k+1}, \; (a^{(n+1)/2}b)^2 \;\rangle \cong \Z \oplus \Z_2 \oplus \Z_{k+1}\]
In particular, this highlights the requirement for both $n$ and $k$ to be odd.


\section{The norm of the homology classes of $\manifold_{k,n}$}
\label{sec:mainresult}

In this section we determine the norm of the non-trivial classes in $H_2 (\manifold_{k,n},\mathbb{Z}_2)$. We first provide an overview. Our proofs of correctness of the calculation rely on technical results by Jaco and Sedgwick~\cite{Jaco-decision-2003} for normal surfaces with boundary on a single 2--triangle torus and Bachman, Derby-Talbot and Sedgwick~\cite{bachman16complicatedHeegaardSplittings} for normal surfaces with boundary on two 2--triangle tori.

We consider a triangulation $\tri'$ of the 3--manifold obtained by removing open neighbourhoods of the blue and red cusps of the complement of $8^3_9$ (\Cref{subsec_hom:triangulation}). This has one cusp and two boundary components, $\partial_1$ (the boundary of the blue cusp) and $\partial_2$ (the boundary of the red cusp). The interior of the identification space $\overline{\tri}$ of $\tri'$ is homeomorphic to the complement of $8^3_9.$ 
Using the natural geometric framing on $\partial_1$ and $\partial_2$ and the surgery description of $\manifold_{k,n}$, we determine how to complete $\tri'$ to a triangulation $\tri'_{k,n}$ of $\manifold_{k,n}$ by gluing suitable layered solid tori to $\partial_1$ and $\partial_2$ (\Cref{subsec_hom:framing}).

Each non-trivial class $\class\in H_2(\manifold_{k,n},\mathbb{Z}_2)$ associates with each edge in $\partial_1$ and $\partial_2$ a parity, which we call its \textbf{boundary pattern} (\Cref{subsec_hom:patterns}). This boundary pattern encodes the parities of the number of points in the intersection of the edges with a surface that is transverse to the triangulation and represents the class $\class .$ There is a taut normal surface $S$ in $\tri'_{k,n}$  that represents the class $\class$ and its intersection with each of $\partial_1$ and $\partial_2$ consists of at most one curve, and these curves of intersection are essential ( \Cref{subsec_hom:surfaces}).

The norm of $\class$ is the maximum Euler characteristic of any normal surface in $\tri'_{k,n}$ representing $\class .$ 
The computation of the norm of $\class$ thus splits into a computation of all normal surfaces of $\tri'$ that have the correct boundary pattern (\Cref{subsec_hom:surfaces}) and determining how these normal surfaces extend into the layered solid tori attached to $\tri'$ along $\partial_1$ and $\partial_2$ (\Cref{subsec_hom:norm}). Normal surfaces in these layered solid tori are well-known (\Cref{subsec_hom:LST}). Moreover, the Euler characteristics of the latter surfaces are determined by their boundary slopes via the Bredon-Wood formula (see \Cref{prop:boundingcurve} or \cite[Corollary 2.2]{Frohman-one-sided-1986}). The conclusion of these computations in \Cref{subsec_hom:norm} is that:
\[
||\;\class_1\;|| = \frac{k+1}{2}, \qquad
||\;\class_2\;|| = \frac{n+1}{2}, \qquad 
||\;\class_3\;|| = \frac{n+k-2}{2} 
\]
In particular, this confirms that the surfaces $\surface_1, \surface_2$ and $\surface_3$ described in \Cref{sec:triangulations} are taut.
We now supply the missing details.


\subsection{A triangulation with two boundary components and one ideal vertex} 
\label{subsec_hom:triangulation}

Starting with the ideal $3$-cusped triangulation $N$ of the complement of $8^3_9$, see \Cref{fig:ideal triangulation} and \Cref{fig:tableNk}, we cut along the vertex links of the blue and red vertex respectively to obtain a $17$-tetrahedra triangulation with two boundary components and one ideal vertex. By construction, the interior of the identification space is homeomorphic with the complement of $8^3_9.$
We denote this triangulation by $\tri'$. Its isomorphism signature is

\centerline{\texttt{rfLLHMzLPMwQcddghghjnklomqopqrwgrrgfxrvqdabxs}}

and its gluing table is given in \Cref{fig:tableT}.

  \begin{table}[h]
   \begin{center}
    \begin{tabular}{r|rrrr}
      Tet & $(012)$ & $(013)$ & $(023)$ & $(123)$ \\
      \hline
$\Delta_{0}$&$3 (012)$&$\partial_1$&$2 (023)$&$1 (123)$\\
$\Delta_{1}$&$5 (012)$&$3 (230)$&$4 (023)$&$0 (123)$\\
$\Delta_{2}$&$7 (012)$&$3 (321)$&$0 (023)$&$6 (123)$\\
$\Delta_{3}$&$0 (012)$&$\partial_1$&$1 (301)$&$2 (310)$\\
$\Delta_{4}$&$8 (012)$&$7 (230)$&$1 (023)$&$6 (023)$\\
$\Delta_{5}$&$1 (012)$&$7 (103)$&$9 (023)$&$6 (310)$\\
$\Delta_{6}$&$10 (012)$&$5 (321)$&$4 (123)$&$2 (123)$\\
$\Delta_{7}$&$2 (012)$&$5 (103)$&$4 (301)$&$11 (123)$\\
$\Delta_{8}$&$4 (012)$&$13 (013)$&$9 (021)$&$12 (123)$\\
$\Delta_{9}$&$8 (032)$&$10 (230)$&$5 (023)$&$13 (210)$\\
$\Delta_{10}$&$6 (012)$&$ 14 (013)$&$9 (301)$&$11 (120)$\\
$\Delta_{11}$&$10 (312)$&$12 (021)$&$14 (201)$&$7 (123)$\\
$\Delta_{12}$&$11 (031)$&$16 (013)$&$15 (023)$&$8 (123)$\\
$\Delta_{13}$&$9 (321)$&$8 (013)$&$14 (032)$&$16 (123)$\\
$\Delta_{14}$&$11 (230)$&$10 (013)$&$13 (032)$&$15 (210)$\\
$\Delta_{15}$&$14 (321)$&$\partial_2$&$12 (023)$&$16 (012)$\\
$\Delta_{16}$&$15 (123)$&$12 (013)$&$\partial_2$&$13 (123)$\\
  \hline
    \end{tabular}
   \end{center}
   \caption{Gluing table for triangulation $\tri'$ of the complement of $8^3_9$ with two real boundary components and one ideal vertex. \label{fig:tableT}}
  \end{table}

In particular, the two boundary components $\partial_1$ (bounding a neighbourhood of the blue cusp in \Cref{fig:ideal triangulation} -- one of the non-symmetric components of $8^3_9$) and $\partial_2$ (bounding a neighbourhood of the red cusp in \Cref{fig:ideal triangulation} -- the symmetric component of $8^3_9$) are two $1$-vertex tori made up of triangles $0(013)$ and $3(013)$, and $15(013)$ and $16(023)$ respectively. We have edge classes

\begin{align*}
e_0 = & \Delta_3(01)=\Delta_0(01), \\
e_2 = & \Delta_0(03)=\Delta_2(03) = \Delta_3(31), \\
e_4 = & \Delta_3(30)=\Delta_1(13) = \Delta_0(13)
\end{align*}

for $\partial_1$ and 

\begin{align*}
e_{18} = & \Delta_{16}(03)=\Delta_{12}(03)=\Delta_{15}(03), \\
e_{19} = & \Delta_{15}(01)=\Delta_{14}(32) = \Delta_{13}(23)=\Delta_{16}(23), \\
e_{20} = & \Delta_{16}(02)=\Delta_{15}(13)
\end{align*}
for $\partial_2$.


\subsection{Boundary patterns} 
\label{subsec_hom:patterns}

By construction, a triangulation $\tri'_{k,n}$ of $\manifold_{k,n}$ can be obtained from $\tri'$ by gluing two suitable layered solid tori to $\partial_1$ and $\partial_2$. Let $\class_1,\class_2,\class_3 \in H_2(\manifold_{k,n},\mathbb{Z}_2)$ be the three non-zero $\mathbb{Z}_2$-- homology classes corresponding to normal surfaces $\surface_1$, $\surface_2$, and $\surface_3$ in $\tri_{k,n}$ (see \Cref{sec:triangulations}).

The fundamental group of the complement of $8^3_9$ is generated by three meridians. We choose the base point for the fundamental group of $\overline{\tri}$ at the vertex of $\partial_1.$ 
Since $\mathbb{Z}_2$ is abelian, 
a homomorphism $\pi_1(\overline{\tri}) \to \mathbb{Z}_2$ uniquely associates labels in $\mathbb{Z}_2$ to the edges of $\partial_1$ and $\partial_2.$ 
Via the inclusion $\overline{\tri} \subset \overline{\tri'_{k,n}}$,
a class in $H_2(\manifold_{k,n},\mathbb{Z}_2)$ determines a homomorphism $\pi_1(\overline{\tri}) \to \mathbb{Z}_2$ mapping the meridian of the cusp of $\overline{\tri}$ to zero.

Conversely, any homomorphism $\pi_1(\overline{\tri}) \to \mathbb{Z}_2$ mapping the meridian of the cusp of $\overline{\tri}$ to zero uniquely determines a class in $H_2(\manifold_{k,n},\mathbb{Z}_2)$ due to the surgery description. 

Hence each $\class_i$ associates labels in $\mathbb{Z}_2$ to the edges in $\partial_1$ and $\partial_2$ (and in turn is uniquely determined by these).
We call these labels the \textbf{boundary pattern} of $\class_i.$ 


\subsection{The normal surfaces of $\tri'$} 
\label{subsec_hom:surfaces}

Let $U_i\subset \manifold_{k,n}$, $[U_i]=\class_i$, $i=1,2,3$ be a taut representative of $\class_i$ in $\manifold_{k,n}$. 
Since $U_i$ is incompressible and $\manifold_{k,n}$ is irreducible, $U_i$ is isotopic to a normal surface. Hence we may assume that
$U_i$ is normal. We further assume that amongst all taut normal surfaces representing $\class_i$, $U_i$ minimises the total edge weight on $\partial_1$ and $\partial_2.$ 

We claim that this implies that 
$U_1$ meets $\partial_1$ in one essential curve and is disjoint from $\partial_2$;
$U_2$ meets $\partial_2$ in one essential curve and is disjoint from $\partial_1$; and 
$U_3$ meets both $\partial_1$ and $\partial_2$ in one essential curve.
Indeed, if there are trivial curves in the intersection with $\partial_1$ or $\partial_2$, then one may perform isotopies across balls to reduce the total weight on $\partial_1$ and $\partial_2.$ The resulting surface may not be normal, but normalises to a normal surface that is still taut and has smaller total edge weight on $\partial_1$ and $\partial_2$ than the initial surface. Similarly, if there are two parallel essential curves in the intersection with $\partial_1$ or $\partial_2$, then one may perform an annular compression reducing the total edge weight on $\partial_1$ and $\partial_2.$ The resulting surface has the same Euler characteristics, and hence still has a taut component that therefore normalises to a normal surfaces of smaller total weight on $\partial_1$ and $\partial_2$ than the initial surface.
It follows that there is at most one essential curve in the intersection with each $\partial_1$ and $\partial_2,$ and the claimed numbers now follow from the boundary pattern.

Let $\normmin_i = U_i \cap \tri'$. We now show that we may assume that $\normmin_1$ and $\normmin_2$ are fundamental surfaces in $\tri.$

Since $\normmin_1$ and $\normmin_2$ have one essential curve on a $1$--vertex torus boundary component, it follows from \cite[Proposition 3.7]{Jaco-decision-2003}, that there exists a fundamental normal surface with the same boundary curve and with equal or larger Euler characteristic. Using Regina \cite{regina} we find that $\tri'$ has $900$ fundamental surfaces in standard coordinates.

First, note that neither $\tri'$ nor the filling tori contain closed normal surfaces of positive Euler characteristic.
Hence, to find the norm of $\class_1$ (resp.\thinspace $\class_2$) it is sufficient to take the minimum over all fundamental normal surfaces $F$ in $\tri'$ with boundary a single essential curve in $\partial_1$ (resp.\thinspace $\partial_2$) of the sum of
\begin{enumerate}
\item[(a)] the negative Euler characteristic of $F$, and 
\item[(b)] the negative Euler characteristic of the incompressible surface $F'$ in the filling torus with $\partial F' = \partial F$.
\end{enumerate}

There are only a few candidate surfaces $F$ to consider.

There are $10$ fundamental surfaces matching the boundary pattern of $\class_1$ and having intersection with $\partial_1$ one essential curve and empty intersection with $\partial_2.$ There are two distinct boundary curves amongst these surfaces, and it suffices to pick a fundamental surface of maximal Euler characteristic for each of these curves.

\begin{figure}[htb]
\centerline{\includegraphics[height=3cm]{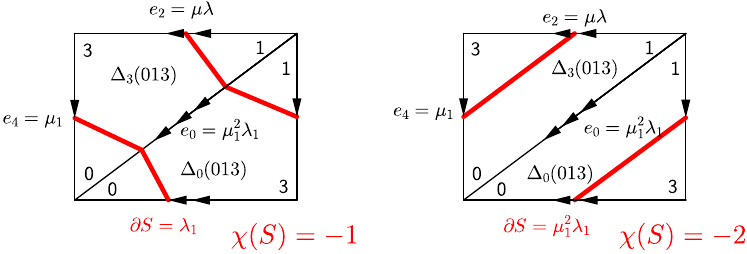}}
\caption{Boundary curves and Euler characteristics of candidate surfaces for $\normmin_1$. \label{fig:c3}}
\end{figure}

Similarly, there are $25$ fundamental surfaces matching the boundary pattern of $\class_2$ and having intersection with $\partial_2$ one essential curve and empty intersection with $\partial_1.$ There are three distinct boundary curves amongst these surfaces, and it suffices to pick a fundamental surface of maximal Euler characteristic for each of these curves; see \Cref{fig:c1}.

\begin{figure}[htb]
\centerline{\includegraphics[height=3cm]{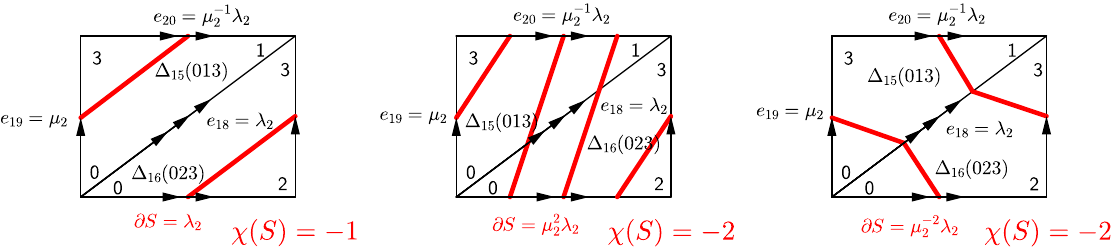}}
\caption{Boundary curves and Euler characteristics of candidate surfaces for $\normmin_2$. \label{fig:c1}}
\end{figure}

The second term is determined by $\partial F' = \partial F$ via the framing and the Bredon-Wood formula, see \Cref{prop:boundingcurve} or \cite[Corollary 2.2]{Frohman-one-sided-1986}. However, we take an equivalent approach by analysing the space of all normal surfaces of the layered solid torus realising the filling of the respective cusp.

We show in \Cref{subsec_hom:norm} that this forces
\[
||\;\class_1\;|| = \frac{k+1}{2} \qquad \text{and}\qquad
||\;\class_2\;|| = \frac{n+1}{2} 
\]
We use these norms to determine the norm of $\class_3.$
The surface $\normmin_3$ has one essential boundary curve in each of $\partial_1$ and $\partial_2$. Since $\normmin_3\cap \partial_j$ is a single essential curve, its Haken summands cannot have any trivial boundary components.
First suppose a Haken sum of fundamental surfaces in $\tri'$ giving $\normmin_3$ contains a surface that is disjoint from one of $\partial_1$ or $\partial_2$. It then follows from \cite[Lemma 2.10]{bachman16complicatedHeegaardSplittings} that there must be exactly one summand 
$H_1$ with $H_1 \cap \partial_1 =  \normmin_3 \cap \partial_1$ and $H_1 \cap \partial_2 = \emptyset$; and 
exactly one summand 
$H_2$ with $H_2 \cap \partial_2 = \normmin_3 \cap \partial_2$ and $H_2 \cap \partial_1 =\emptyset.$ 
Since $\normmin_3$ is the intersection with $\tri'$ of the normal surface $U_3$ in $\tri'_{k,n},$ it follows that each $H_1$ and $H_2$ extends to a normal surface $V_1$ and $V_2$ respectively in $\tri'_{k,n}$ via a subsurface of $U_2.$
The boundary pattern implies that $[V_1] = \class_1$ and  $[V_2] = \class_2$, and since $U_3$ is taut, we have $||\;\class_3\;|| \ge ||\;\class_1\;|| + ||\;\class_2\;||.$
But we already know from the existence of $\surface_3$ from \Cref{sec:triangulations} that 
\[ 
||\;\class_3\;|| \le \frac{n+k-2}{2} < \frac{k+1}{2} + \frac{n+1}{2} = ||\;\class_1\;|| + ||\;\class_2\;||
\]
Hence every Haken summand $H$ with non-empty boundary satisfies $H \cap \partial_1 = \normmin_3 \cap \partial_1$ and $H \cap \partial_2 = \normmin_3 \cap \partial_2.$ Since there is no closed normal surface with positive Euler characteristic, we may assume that every Haken summand in a Haken sum giving $\normmin_3$ meets each boundary component of $\tri'$ in non-trivial essential curves.

After sorting out duplicates and non-maximal Euler characteristic examples, there are $7$ compatibility classes of normal surfaces in $\tri'$ representing $\class_3.$ In particular, the first class can be reduced to a single fundamental surface, and each of the remaining six classes can be reduced to Haken sums of two fundamental surfaces, i.e.\thinspace Haken sums of the form $w_1H_1 + w_2H_2$. For all pairs $H_1,$ $H_2$ and their possible Haken sums $w_1H_1 + w_2H_2$, we list in \Cref{fig:compcl}
their edge weights on $\partial_1$ and $\partial_2$, their Euler characteristics, and weights $w_1$ and $w_2.$

\begin{table}
\begin{center}
\begin{tabular}{l||l|l|l||l|l|l||l|l}
$\#$&$w(e_0)$&$w(e_2)$&$w(e_4)$&$w(e_{18})$&$w(e_{19})$&$w(e_{20})$&$\chi$&weight \\
\hline
\hline
$1$&$2$&$1$&$1$&$0$&$1$&$1$&$-2$&$1$\\
\hline
\hline
$2$&$2$&$0$&$2$&$2$&$1$&$1$&$-2$&$2k+1$\\
&$2$&$1$&$3$&$2$&$2$&$0$&$-3$&$2l+1$\\
\hline
&$4(k+l+1)$&$2l+1$&$4k+6l+5$&$4(k+l+1)$&$2k+4l+3$&$2k+1$&$-4k-6l-5$\\
\hline
\hline
$3$&$2$&$0$&$2$&$2$&$1$&$1$&$-2$&$2k+1$\\
&$2$&$1$&$1$&$2$&$0$&$2$&$-1$&$2l+1$\\
\hline

&$4(k+l+1)$&$2l+1$&$4k+2l+3$&	$4(k+l+1)$&$2k+1$&$2k+4l+3$&	$-4k-2l-3$\\
\hline
\hline
$4$&$0$&$1$&$1$&$0$&$1$&$1$&$-1$&$2k+1$\\
&$2$&$2$&$0$&$2$&$1$&$3$&$-2$&$2l$\\
\hline

&$4l$&$2k+4l+1$&$2k+1$&	$4l$&$2k+2l+1$&$2k+6l+1$&	$-2k-4l-1 $\\
\hline
\hline
$5$&$2$&$1$&$1$&$2$&$0$&$2$&$-1$&$2k+1$\\
&$2$&$2$&$0$&$2$&$1$&$3$&$-2$&$2l+1$\\
\hline

&$4(k+l+1)$&$2k+4l+3$&$2k+1$&	$4(k+l+1)$&$2l+1$&$4k+6l+5$&	$-2k-4l-3 $\\
\hline
\hline
$6$&$0$&$1$&$1$&$0$&$1$&$1$&$-1$&$2k+1$\\
&$2$&$1$&$3$&$2$&$2$&$0$&$-3$&$2l$\\
\hline

&$4l$&$2k+2l+1$&$2k+6l+1$&	$4l$&$2k+4l+1$&$2k+1$&	$-2k-6l-1$\\
\hline
\hline
$7$&$0$&$1$&$1$&$0$&$1$&$1$&$-1$&$2k+1$\\
&$2$&$3$&$1$&$2$&$2$&$4$&$-3$&$2l$\\
\hline

&$4l$&$2k+6l+1$&$2k+2l+1$&	$4l$&$2k+4l+1$&$2k+8l+1$&	$-2k-6l-1$\\

\hline
\end{tabular}
\end{center}
\caption{Compatibility classes of candidates for a norm minimiser for $\class_3$. 
\label{fig:compcl}
}
\end{table}


\subsection{Some basic facts on layered solid tori} 
\label{subsec_hom:LST}

Layered solid tori denote a two parameter family $\lst(j,k)$ of $1$-vertex triangulations of the solid torus, see \cite{Jaco-layered-2006} for details. 
Given $\lst(j,k)$ ($j$ and $k$ co-prime), the boundary of the meridian disk has edge weights $j$, $k$, and $j+k$ on the three boundary edges.

In this section we are only interested in $\lst(0,1)$, and the family $\lst(1,m)$, $m \geq 2$.
The layered solid torus $\lst(0,1)$ has three tetrahedra. One of its edges is parallel to the boundary of its meridian disk (the unique edge with edge weight $0$). We call this edge the \textbf{meridional edge} of the layered solid torus.
The layered solid torus $\lst(1,m)$ is obtained inductively from the $1$-tetrahedron layered solid torus $\lst(1,2)$ by layering a tetrahedron on the edge of edge weight $m$ of $\lst(1,m-1).$ In particular, for $m \ge 2$ there is a unique boundary edge with edge weight $1$, and we call this the \textbf{longitudinal edge}.

It follows from \cite[Theorem 5.3]{Jaco-layered-2006} that the only normal surfaces in $\lst(1,m)$ with connected, essential boundary are the meridian disk (with edge weights $1$, $m$, and $m+1$ in the boundary, case (3) in the theorem), and one-sided incompressible normal surfaces with boundary slope what is called the slope of an even ordered edge in $\lst(1,m)$ (this is an edge in $\lst(1,m)$ that represents an even multiple of a generator of the fundamental group of the torus, case (6) in the theorem).

For $\lst(1,m)$ the slopes of even ordered edges have edge weights $1$, $m-2$, and $m-1$ (a M\"obius strip), $1$, $m-4$, and $m-3$ (a surface of Euler characteristic $-1$), all the way down to $1$, $1$, and $0$ for $m$ odd (a surface of Euler characteristic $\frac{1-m}{2}$) and all the way down to $1$, $0$, and $1$ for $m$ even (a surface of Euler characteristic $\frac{2-m}{2}$).
In particular, we have the following:

\begin{corollary}[Theorem 5.3 from \cite{Jaco-layered-2006}]
  \label{cor:lst}
 Let $S$ be a normal surface in $\lst(1,m)$, $m \geq 2$, with connected essential boundary. Then $S$ has edge weight $1$ on the longitudinal edge of $\lst(1,m).$
\end{corollary}

If the layered solid torus triangulation is generated using the standard Regina~\cite{regina} function, then the surface $S$ in \Cref{cor:lst} has edge weight $1$ on the edge $e_2 = \Delta_0 (03) = \Delta_0 (21)$.


\subsection{Determining the natural geometric framing on $\partial_1$ and $\partial_2$} 
\label{subsec_hom:framing}

Just as in \Cref{sec:manifolds}, we denote the geometric framing of $\partial_i$ by $(\meridian_i,\longitude_i)$, $i=1,2$.

Gluing $\lst(0,1)$ to $\partial_2$ with the meridian edge along $e_{19}$ produces a manifold with fundamental group free of rank $2$ and hence the $2$-component unlink. Hence, $e_{19}$ runs parallel to the meridian $\meridian_2$ in the geometric framing of $\partial_2$. Gluing in the same layered solid torus along the other two edges demonstrates that $e_{18}$ runs parallel to the geometric longitude $\longitude_2$ and that $e_{20}$ runs parallel to $\meridian^{-1}_2\longitude_2$.

Gluing $\lst(0,1)$ with the meridian edge onto the edges $e_4$ and $e_2$ yields the fundamental group of the link complement of the remaining two link components (a double twisted Hopf link). Gluing the meridian edge onto $e_0$ shows that this edge runs parallel to $\meridian_1^{2}\longitude_1$ in the geometric framing. This implies that $e_4$ runs parallel to $\meridian_1^{-1}$ and $e_2$ runs parallel to $\meridian_1\longitude_1$. This computation is confirmed by noting that the homological longitude has edge weight $2$ on $e_0$ and edge weight $1$ on $e_2$ and $e_4$ (as observed by an orientable surface with this boundary curve, see \Cref{fig:c3} on the left) and the fact that this homological longitude coincides with the geometric longitude.

The above computations were verified using SnapPy \cite{SnapPy} and SnapPy in sage for the homological longitude. See \Cref{fig:c1,fig:c3} for pictures of $\partial_1$ and $\partial_2$ with their respective framings.

The information about the natural geometric framing on $\partial_1$ and $\partial_2$ combined with  \Cref{eq:snappy} determine which layered solid tori to glue to $\partial_1$ and $\partial_2$ to recover $\manifold_{k,n}$. 

In detail, $\lst(1,k-1)$ must be glued to $\partial_1$ with the edge weight-$1$ boundary edge of the layered solid torus glued to $e_2$. Similarly, $\lst(1,n)$ must be glued to $\partial_2$ with the edge weight-$1$ boundary edge of the layered solid torus glued to edge $e_{19}$. See \Cref{fig:gluing} for a detailed picture of the gluings.

\begin{figure}[htb]
\centerline{\includegraphics[height=10cm]{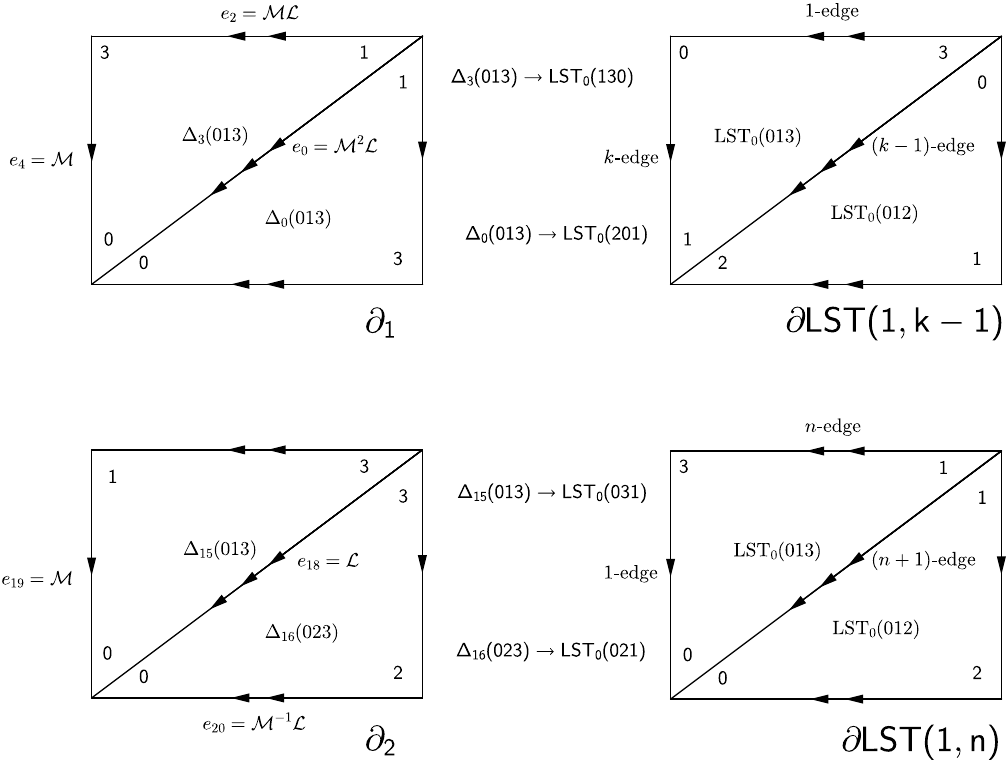}}
\caption{Filling $\tri'$ with suitable layered solid tori. \label{fig:gluing}}
\end{figure}


\subsection{Euler characteristic of the completion of $\normmin_i$, $1\leq i \leq 3$.} 
\label{subsec_hom:norm}

We are now in a position to determine the norms of $\class_i$, $ 1\leq i \leq 3$.

The triple $(a,b,c)$ denotes the (parities of the) edge weights of $e_0$, $e_2$, and $e_4$ in that order for $\partial_1$ and $(d,e,f)$ denotes the (parities of the) edge weights of $e_{18}$, $e_{19}$, and $e_{20}$ in that order for $\partial_2$.

The boundary pattern of $\class_1$ is $(0,1,1)$ with respect to the three edges on $\partial_1$ and  $(0,0,0)$ with respect to $\partial_2.$ These parities are realised by the following edge weights of the fundamental surfaces. 
We have $(2,1,1)$ ($\chi = -1$), and $(0,1,1)$ ($\chi = -2$). The former cannot be continued into the layered solid torus, the latter continues along the canonical surface of Euler characteristic $\frac{-k+3}{2}$. In total, this yields a surface of Euler characteristic $\frac{-k+3}{2} - 2 = \frac{-k-1}{2}$, hence the norm of $\class_1$ is $\frac{k+1}{2}$.  

\medskip

The boundary pattern of $\class_2$ is $(0,1,1)$ with respect to the three edges on $\partial_2$ and  $(0,0,0)$ with respect to $\partial_1.$ These parities are realised by the following edge weights of the fundamental surfaces. We have $(0,1,1)$ ($\chi = -1$), $(2,1,3)$ ($\chi = -2$), and $(2,1,1)$ ($\chi = -2$). The first can be completed with canonical surface of Euler characteristic $\frac{-n+1}{2}$. In total, this produces a surface representing $\class_2$ with Euler characteristic $\frac{-n+1}{2} - 1 = \frac{-n-1}{2}$. The second cannot be continued into the layered solid torus. The third bounds a normal surface in the layered solid torus obtained from the canonical normal surface by one boundary compression. It is hence of Euler characteristic $\frac{-n+3}{2}$ yielding, again, a normal surface of total Euler characteristic $\frac{-n+3}{2} - 2 = \frac{-n-1}{2}$. Hence, both the first and the third surfaces are taut, and we thus have that the norm of $\class_2$ is $\frac{n+1}{2}$.

\medskip

As shown in \Cref{subsec_hom:surfaces}, having computed the claimed norms of $\class_1$ and $\class_2,$ we may now restrict ourselves to normal surfaces with the boundary pattern of $\class_3.$
Since $e_2$ and $e_{19}$ are both glued to edge weight $1$-edges of layered solid tori of type $\lst(1,m)$, it follows from \Cref{cor:lst} that a necessary condition for a normal surface in $\tri'$ to extend into the layered solid tori is that it has  edge weight $1$ on $e_2$ and $e_{19}.$

Going through \Cref{fig:compcl}, compatibility class by compatibility class, this leaves us with the following cases:
\begin{description}
  \item[Class 1] Edge weights $1$ for both $e_2$ and $e_{19}$, but cannot be extended into layered solid torus at $\partial_1$ (no normal surface in $\lst(1,k-1)$ has edge weights $(2,1,1)$).
  \item[Class 2] Has $w(e_{19}) > 1$ for all of its members and can thus be disregarded.
  \item[Class 3] Has valid edge weights for $k=l=0$. This yields a surface with Euler characteristic $-3$ and edge weights $(4,1,3)$ on both $\partial_1$ and $\partial_2$. While this surface extends into $\lst(1,n)$ at $\partial_2$, it does not extended into $\lst(1,k-1)$ at $\partial_1$.
  \item[Class 4] Has valid edge weights for $k=l=0$. This yields a surface with Euler characteristic $-1$ and edge weights $(0,1,1)$ on both $\partial_1$ and $\partial_2$. This surface extends via the canonical $0/1$-surface into both $\lst(1,n)$ and $\lst(1,k-1)$ and produces a closed $1$-sided surface of Euler characteristic
  \[-1 + \frac{1-n}{2} +  \frac{3-k}{2} = - \frac{n+k-2}{2}. \]
  \item[Class 5] Has $w(e_{2}) > 1$ for all of its members and can thus be disregarded.
  \item[Class 6] Has valid edge weights for $k=l=0$. Gives same Euler characteristic and edge weights as Class 4.
  \item[Class 7] Has valid edge weights for $k=l=0$. Gives same Euler characteristic and edge weights as Class 4.
\end{description}

Altogether, it follows that the norm of $\class_3$ equals $\frac{n+k-2}{2}$.


\section{Generalisations}
\label{sec:constructions}

As mentioned in the introduction, the triangulations $\tri_{k,n}$ can be generalised to further conjecturally minimal triangulations. In this section we present one such family that was again found using experimentation with the census.

Consider the triangulation $\mathcal{U}_{3,3}$ with Regina isomorphism signature \texttt{iLLwQPcbeefgehhhhhqhhqhqx}, and given by the gluing table

\begin{center}
\begin{tabular}{r|rrrr}
Tet&$(012)$&$(013)$&$(023)$&$(123)$\\
\hline
$0  $&$3 (012)$&$1 (102)$&$2 (023)$&$1 (123)$\\
$1  $&$0 (103)$&$4 (102)$&$4 (023)$&$0 (123)$\\
$2  $&$6 (012)$&$5 (013)$&$0 (023)$&$4 (301)$\\
$3  $&$0 (012)$&$4 (231)$&$6 (132)$&$5 (032)$\\
$4  $&$1 (103)$&$2 (231)$&$1 (023)$&$3 (301)$\\
$5  $&$7 (103)$&$2 (013)$&$3 (132)$&$7 (123)$\\
$6  $&$2 (012)$&$7 (320)$&$7 (201)$&$3 (032)$\\
$7  $&$6 (230)$&$5 (102)$&$6 (310)$&$5 (123)$
\end{tabular}
\end{center}

This triangulation coincides with census triangulation \texttt{t12546} $\#1$ of a once-cusped hyperbolic $3$-manifold. Its first homology group is $\mathbb{Z} \oplus \mathbb{Z}_2 \oplus \mathbb{Z}_4 $, and its hyperbolic volume is approximately $7.555$. 

The triangulation contains two copies of solid torus $T_3$ (tetrahedra $0$, $1$, $4$, and tetrahedra $5$, $6$, $7$). But unlike in $\tri_{3,3}$, in $\mathcal{U}_{3,3}$ they are not identified along their boundaries. Instead, they attach to a central complex consisting of tetrahedra $2$, and $3$.

We replace the two copies of $T_3$ in $\mathcal{U}_{3,3}$ by $T_k$ and $T_n$, $k,n\geq 3$ odd, to obtain a family of triangulations $\mathcal{U}_{k,n}$ with $n+k+2$ tetrahedra, and three non-trivial $\mathbb{Z}_2$--torsion classes. The three normal surfaces $F_1$, $F_2$, and $F_3$ of each solid torus pair up symmetrically through the central $2$--tetrahedra interface. This produces three normal surfaces, each consisting entirely of quadrilaterals. Because of the symmetric pairing of surfaces from the solid tori, $\mathcal{U}_{k,n}$ is isomorphic to $\mathcal{U}_{n,k}$.
The three quadrilateral normal surfaces have Euler characteristic $-\frac{n+k}{2}$, $-\frac{n+k}{2}$, and $-2$ respectively. Their sum, hence, equals the negative of the number of tetrahedra in $\mathcal{U}_{k,n}$. We conjecture that these surfaces are taut, and hence these triangulations are minimal due to \Cref{thm:sumofnorms}.

As with the main examples of this paper, the family $\mathcal{U}_{k,n}$ arises by Dehn surgery on a thrice-cusped hyperbolic 3--manifold. Indeed, replacing the two solid tori with a cusp each yields a $3$-cusped ideal triangulation with Regina isomorphism signature \texttt{kLLPwLQkceefeijijijiiapuuxptxl}. According to SnapPy's identify function \cite{SnapPy} this triangulation is decomposed into ten regular ideal hyperbolic tetrahedra and gives the thrice-cusped Platonic manifold \texttt{otet10\_00015} (see \cite{Grner2016ACO}). It follows that the triangulation of the Platonic manifold is minimal.

\medskip

It is an interesting open problem to determine further generalisations of this construction, using even more ways of identifying the two solid torus triangulations $T_k$ and $T_n$ along a subcomplex.


\bibliographystyle{plain}
\bibliography{newfamily15}


\address{J. Hyam Rubinstein\\School of Mathematics and Statistics, The University of Melbourne, VIC 3010, Australia\\{joachim@unimelb.edu.au}\\----- }

\address{Jonathan Spreer\\School of Mathematics and Statistics F07, The University of Sydney, NSW 2006 Australia\\{jonathan.spreer@sydney.edu.au\\-----}}

\address{Stephan Tillmann\\School of Mathematics and Statistics F07, The University of Sydney, NSW 2006 Australia\\{stephan.tillmann@sydney.edu.au}}

\Addresses
                                                      
\end{document}